\documentclass{nhmOF}
\usepackage{amsmath}
  \usepackage{paralist}
  \usepackage{graphics} %% add this and next lines if pictures should be in esp format
  \usepackage{epsfig} %For pictures: screened artwork should be set up with an 85 or 100 line screen
\usepackage{graphicx}  \usepackage{epstopdf}%This is to transfer .eps figure to .pdf figure; please compile your paper using PDFLeTex or PDFTeXify.
 \usepackage[colorlinks=true]{hyperref}
\hypersetup{urlcolor=blue, citecolor=red}

\def\currentvolume{X}
 \def\currentissue{X}
  \def\currentyear{200X}
   \def\currentmonth{XX}
    \def\ppages{X--XX}
     \def\DOI{10.3934/nhm.2022015}

\newtheorem{theorem}{Theorem}[section]

\newtheorem{lemma}[theorem]{Lemma}

\theoremstyle{definition}
\newtheorem{definition}[theorem]{Definition}
\newtheorem{remark}{Remark}

\newcommand{\ep}{\varepsilon}
\newcommand{\eps}[1]{{#1}_{\varepsilon}}

\usepackage{amsmath,amssymb,graphics,graphicx,epsfig,multirow,color,subfig,amsfonts,amsbsy,mathtools,xargs,tikz,float}
\usepackage{graphicx}
\usepackage{epstopdf} %converting to PDF
\newcommand{\R}{\mathbb{R}}
\newcommand{\N}{{\mathbb N}}
\newcommand{\Z}{{\mathbb Z}}

\newcommand{\PP}{\mathcal{P}}
\newcommand{\M}{\mathcal{M}}
\newcommand{\C}{{\mathcal C}}
\newcommand{\A}{{\mathcal A}}
\newcommand{\F}{{\mathcal F}}
\newcommand{\I}{\mathcal{I}}

\newcommand{\WW}{\mathcal{W}}

\newtheorem{prop}[theorem]{Proposition}
\newtheorem{example}[theorem]{Example}
\usepackage{todonotes}
\def\MyItem[#1]#2{\item[{\rm #1}]#2}

\title[modeling the spread of viral infections] %Use the shortened version of the full title
      {A measure model for the spread of viral infections with mutations}

\author[Xiaoqian Gong and Benedetto Piccoli]{}


 \email{xiaoqian.gong@asu.edu}
 \email{piccoli@camden.rutgers.edu}

\thanks{$^*$ Corresponding author: Benedetto Piccoli}

\begin{document}
\maketitle

% Enter the first author's name and address:
\centerline{\scshape Xiaoqian Gong}
\medskip
{\footnotesize
% please put the address of the first author
 \centerline{School of Mathematical and Statistical Science}
   \centerline{Arizona State University}
   \centerline{Tempe, AZ, 85281, USA}
} % Do not forget to end the {\footnotesize by the sign }

\medskip

\centerline{\scshape Benedetto Piccoli$^*$}
\medskip
{\footnotesize
 % please put the address of the second  and third author
 \centerline{Department of Mathematical Sciences and Center for Computational and Integrative Biology}
   \centerline{Rutgers University}
   \centerline{Camden, NJ, 08102, USA}
}

\bigskip

% The name of the associate editor will be entered by an editorial staff
% "Communicated by the associate editor name" is not needed for special issue.
% \centerline{(Communicated by the associate editor name)}

%The abstract of your paper
\begin{abstract}
Genetic variations in the COVID-19 virus are one of the main causes of the COVID-19 pandemic outbreak in 2020 and 2021. In this article, we aim to introduce a new type of model, a system coupled with ordinary differential equations (ODEs) and measure differential equation (MDE), stemming from the classical SIR model for the variants distribution. Specifically, we model the evolution of susceptible $S$ and removed $R$ populations by ODEs and the infected $I$ population by a MDE comprised of a probability vector field (PVF) and a source term. In addition, the ODEs for $S$ and $R$ contains terms that are related to the measure $I$. We establish analytically the well-posedness of the coupled ODE-MDE system by using generalized Wasserstein distance. We give two examples to show that the proposed ODE-MDE model coincides with the classical SIR model in case of constant or time-dependent parameters as special cases.

\end{abstract}

%The title of your section 1
\section{Introduction}

The 2020 COVID-19 pandemic generated renewed
interests in epidemiological models
for infectious diseases. Researchers and modelers from different areas including applied math, physics, public health, engineering and others, developed
many different approaches depending on the final aim,
which included nowcasting and possible scenarios,
prediction of the pandemic evolution,
evaluation of lock down and social distancing
measures as well as economic impact.

The reasons for the pandemic outbreak included
difficulty in detection of the virus for
asymptomatic, fast spread due to globalization
and emergence of different variants usually
associated with a specific country where they
were first observed. The latter include:
B.1.1.7 initially detected in UK,
B.1.351 detected in South Africa,
P.1 detected in Japan for travelers from Brazil,
and  B.1.427 and B.1.429 identified in California. The name variant is
a convenient way to represent a family of mutations grouped by genetic similarities. The total number of detected mutations so far exceeds the two millions, see
\cite{Kupferschmidt844}.

Our interest is in introducing a new type
of models stemming from the classical SIR model introduced in the pioneering work of
Kormack and McKendrick \cite{kermack1932contributions},
coupled with a new type of differential equations
for measures, called Measure Differential Equations
(briefly MDE) introduced in \cite{Piccoli-ARMA}
for the variants distribution among infected, which
is connected to the concept of Probability Vector Field
(briefly PVF).

Let us first recall that the SIR model
was generalized in number of recent papers
and in various directions, such as:
1) using time-dependent parameters as
infectivity rates \cite{chen2020time,ruktanonchai2020assessing};
2) adding population compartments for different
stages of the disease
\cite{giordano2020modelling};
3) increasing complexity with age-structure and spatial models \cite{britton2020mathematical,colombo2020well,colombo2020age,Chyba2020a,zhang2020evolving,zhang2020changes};
4) using multiscale approaches and infinite dimensional systems \cite{anita2017reactiondiffusion,bellomo2020multi,keimer2020modeling}.
The different type of models can be used
for some of the scopes discussed above
\cite{metcalf2020mathematical,vespignani2020modelling},
but the difficulty in tuning with data
and use for prediction is known since long time
\cite{Hyman17}.

Let us go back to our model to include virus variants
in the SIR model. The evolution of susceptible $S$ and removed $R$ populations will still be detected by an Ordinary Differential
Equation (briefly ODE). On the other side,  there are more than two million SARS-COV2 genetic variations, see the figures on pages 844/845 in \cite{Kupferschmidt844}. Branching out Each dot represents a virus isolated from a COVID-19 patient in
this family tree of SARS-COV2, which shows a tiny subset of the more
than 2 million viruses sequenced so far. The World Health Organization
currently recognizes four variants of concern and four variants of interest.
We assume that the virus variants are captured by a continuous variable $\alpha\in\R$ and the infected
population is represented by a Radon measure $I$ on $\R$ with finite mass. In simple words the value $I(A)$, $A\subset\R$,
represents the number of infected people having a virus variant
corresponding to parameters $\alpha\in A$.
The dynamics of $I$ will be captured by an MDE with two
components: a finite-diffusion term, which represents
the emergence of variants in infected people, and a source
term, which represents the inflow of susceptible getting
infected and the outflow of infected to removed.
On the other side, the ODE for $S$ and $R$ will contain terms
which depend on the measure $I$ corresponding to the
inflow and outflow term of the MDE.
More precisely, the original SIR dynamics is modified
since the classical infection rate $\beta$ is function
of the parameter $\alpha$ identifying the virus variant,
and the same occur for the recovery rate $\nu$.
Therefore, the resulting dynamics is a systems of fully coupled
ODE-MDE.

To deal with the introduced model, we resort to recent
results on MDEs \cite{Piccoli-ARMA} and MDE with sources \cite{PR19}. We first recall some basic tools
for measures, including the Wasserstein distance and
the generalized Wasserstein distance (since we deal
with measures with variable mass).
Then we provide a definition of solution and Lipschitz semigroup
of solutions for the coupled ODE-MDE system, where the MDE is comprised of a PVF and a source term.
The main result of the paper is the existence of a Lipschitz
semigroup under suitable assumption. To state the assumptions
we have to deal with a space $\R^m\times \M$, where
$\M$ indicates the space of Radon measures with finite mass
and compact support endowed with the generalized
Wasserstein distance. The conditions for the existence
of the Lipschitz semigroup are the natural generalization
of the conditions for ODE and MDE and include:
sub-linear growth of supports for the MDE, Lipschitz continuity of the vector field of the ODE and the PVF of the MDE, and, finally,
uniform boundedness and Lipschitz continuity of the source of the MDE. Notice that the uniform boundedness
of the source is chosen to simplify some proofs, but can
be relaxed by sub-linear growth conditions.
A Lipschitz semigroup can then be selected, to achieve
uniqueness, by prescribing small-time evolution
of finite sums of Dirac masses.
The ODE-MDE system representing the SIR model with virus
variants is then shown to satisfy the assumption for the
existence of a Lipschitz semigroup of solutions.

The paper is organized as follows: We first recall basic definitions and results on generalized Wasserstein distance and measure differential equations in Section $2$. Then we provide existence and uniqueness results for a system of coupled ODE-MDE in Section $3$. In the last Section, we introduce a measure model for viral infections with mutations consisting of one MDE for infected coupled with a system of two differential equations for susceptible and removed. Two examples show how the model coincides with the original SIR model in case of constant parameters (not depending on mutation)
and includes SIR model with time-dependent parameters as a special case.

\section{Basic definition and results}
%For simplicity we restrict to $\R^n$, but a local theory can be easily developed  for manifolds admitting a partition of unity.
We use the symbol $|\cdot|$ to indicate the Euclidean norm, and
for every $R>0$, $B(0,R)$ for the ball of radius $R$ centered at the origin.
The  symbol $T\R^n$ indicates the tangent bundle of $\R^n$,
$\pi_1:T\R^n\to \R^n$ is the base-projection $\pi_1(x,v)=x$ and
$\pi_{13}:(T\R^n)^2\mapsto (\R^n)^2$ is given by
$\pi_{13}(x,v,y,w)=(x,y)$.
For every $A\subset \R^n$, $\chi_A$ indicates the characteristic function of the set $A$ and ${\C}^\infty_c(\R^n)$ the space of compactly supported smooth functions on $\R^n$.

Given $(X,d)$ Polish space (complete separable metric space)
we indicate by $\M(X)$ the set of positive Borel measures
with finite mass and compact support, by $\PP(X)$ the set of probability measures and by $\PP_c(X)$ the set of probability measures with compact support on $X$. For $\mu\in\M$
we denote by $|\mu|=\mu(X)$ the total mass
and by $Supp(\mu)$ its support.
Given $(X_1,d_1)$, $(X_2,d_2)$ Polish spaces, $\mu\in\PP(X_1)$, $\phi:X_1\to X_2$ measurable and Borel set $A \subset X_2$, we set
$\phi\#\mu\in\PP(X_2)$ by $\phi\#\mu(A) = \mu ( \phi^{-1}(A) ) =
\mu(\{x\in X_1:\phi(x)\in A\})$. Consider a measure $\mu\in\M(X_1)$, a family of measures $\nu_x\in\M(X_2)$, with $x\in X_1$, and a function $\varphi \colon X_1 \times X_2 \to \mathbb{R}$ such that
$v\to \varphi (x,v)\in L^1(d\nu_x)$
for $\mu$-almost every $x$ and
$x\to \int_{X_2} \varphi(x,v) d\nu_x(v)
\in L^1(d\mu)$.
Then we define
$\int_{X_1\times X_2} \varphi(x,v) \ d(\mu\otimes \nu_x) = \int_{X_1} \int_{X_2} \varphi(x,v) d\nu_x(v)\ d\mu(x)$.
For $\mu$, $\nu\in \M(X)$, we denote by $P(\mu,\nu)$ the set of transference plans from $\mu$ to $\nu$,
i.e. the set of probability measures on $X\times X$ with marginals
equal to $\mu$ and $\nu$ respectively.
The cost of a transference plan  $\tau\in P(\mu,\nu)$ is defined as
$J(\tau)=\int_{X^2} d(x,y)\, d\tau(x,y)$
and the Monge-Kantorovich or optimal transport problem amounts to find a cost minimizer. The value of the attained minimum
is called the Wasserstein distance between $\mu$ and $\nu$:
\[
W^X(\mu,\nu)=\inf_{\tau\in P(\mu,\nu)} J(\tau).
\]
In general, a Wasserstein distance $W^p$ can be defined for $p\geq 1$
by setting $J(\tau)=\left(\int_{X^2} d(x,y)^p\, d\tau(x,y)\right)^{\frac{1}{p}}$. We indicate
by $P^{opt}(\mu,\nu)$ the set of optimal transference plans from $\mu$ to $\nu$.

We will consider measures with time-varying total mass, thus
we will consider the
%Since the solution to \eqref{eq:MDE-w-s} may change mass in time, we need to resort to a
generalized Wasserstein distance:
\begin{definition}[The generalized Wasserstein distance]
\label{def: g-wass}
Let $\mu,\nu\in\M(\R)$ be two measures. We define the functional
\begin{equation}
W^g(\mu,\nu):=\inf_{\tilde\mu,\tilde\nu\in\M,\,|\tilde\mu|=|\tilde\nu|}|\mu-\tilde\mu|+|\nu-\tilde\nu|+W(\tilde\mu,\tilde\nu).
\end{equation}
\end{definition}
The generalized Wasserstein distance is thus obtained
combining an $L^1$ or total variation cost for removing/adding
mass and transportation cost for the rest of the mass.
Various properties of the generalized Wasserstein distance
can be found in \cite{GenWass1,GenWass2}.

Lipschitz-type conditions for evolution of measures with time-varying mass will be defined in terms of the following operator.
\begin{definition}[The operator $\mathcal{W}^g$]
\label{def: operator_W_g}
Consider $V_i \in \mathcal{M}(T\mathbb{R}^n)$ and $\tilde{V}_i$ satisfying $\tilde{V}_i \leq V_i, i=1, 2$. Let $\mu_i = \pi_1 \#V_i$ and $\tilde{\mu}_i = \pi_1 \# \tilde{V}_i, i=1, 2$ be the projections on the base space. We define the following non-negative operator
\begin{equation}\label{eq:WW-g}
\begin{aligned}
\WW^g (V_1,V_2) \coloneqq & \inf \left\{
\int_{T\R^n\times T\R^n} |v-w| \ dT(x,v,y,w) \colon \right.\\
& \left.T\in P(\tilde{V}_1,\tilde{V}_2) \text{ with } , \tilde{V}_i \leq V_i, i=1, 2, \pi_{13}\# T\in P^{opt}(\tilde{\mu}_1,\tilde{\mu}_2) \right.\\
& \left. \text{ where } (\tilde{\mu}_1, \tilde{\mu}_2) \text{ is a minimizer in Definition }\ref{def: g-wass} \right\}.
\end{aligned}
\end{equation}
\end{definition}

\subsection{Measure differential equations}
In this section we recall the basic definitions for the
theory of MDEs introduced in \cite{Piccoli-ARMA}. The original results were provided for probability measures but they hold
for measures with constant finite mass with obvious modifications.

These type of equations are based on a generalization of the concept of vector field to measures as follows.
\begin{definition}
A Probability Vector Field (briefly PVF)  on $\M(\R^n)$ is a map
$V: \M(\R^n)\to\PP(T\R^n)$ such that $\pi_1\# V[\mu]=\mu$.
\end{definition}
A Measure Differential Equation (MDE) corresponding to a PVF
$V$ is defined by:
\begin{equation}\label{eq:MDE}
\dot\mu=V[\mu].
\end{equation}
The mass of $\mu$ over a set $A\subset\R^n$ is dispersed
along the velocites of the support of $V[\mu]$ restricted
to $A\times \cup_{x\in A}T_x\R^n$.\\
Given an MDE and $\mu_0\in\M(\R^n)$ we define the Cauchy problem:
\begin{equation}\label{eq:MDE-Cauchy}
\dot\mu=V[\mu],\qquad \mu(0)=\mu_0.
\end{equation}
A solution to (\ref{eq:MDE-Cauchy}) is defined using weak solutions as follows.
\begin{definition}\label{def:sol-MDE}
A solution to (\ref{eq:MDE-Cauchy}) is a map $\mu:[0,T]\to \M(\R^n)$  such that $\mu(0)=\mu_0$, $|\mu(t)|$ is constant and the following holds.
For every $f\in{\C}^\infty_c(\R^n)$,  the integral
$\int_{T\R^n} (\nabla f(x)\cdot v)\ dV[\mu(s)](x,v)$ is defined for almost every $s$,
the map $s\to \int_{T\R^n} (\nabla f(x)\cdot v)\ dV[\mu(s)](x,v)$ belongs to $L^1([0,T])$, and the map $t\to \int f\, d\mu(t)$ is absolutely
continuous and for almost every $t\in [0,T]$ it satisfies:
\begin{equation}\label{eq:sol-MDE}
\frac{d}{dt}\int_{\R^n} f(x)\,d\mu(t)(x) =
\int_{T\R^n} (\nabla f(x)\cdot v)\ dV[\mu(t)](x,v).
\end{equation}
\end{definition}
The existence of solutions to (\ref{eq:MDE-Cauchy}) holds
true under the following assumptions:
\begin{itemize}
\item[ {(H:bound)}] $V$ is support sublinear, i.e. there exists $C>0$ such that
for every $\mu\in\M(X)$ it holds:
\[
\sup_{(x,v)\in Supp(V[\mu])} |v|\leq C
\left( 1 + \sup_{x\in Supp(\mu)} |x|\right).
\]
 {
\item[(H:cont)] the map $V:\M(\R^n)\to \M(T\R^n)$
is continuous (for the topology given by the Wasserstein metrics $W^{\R^n}$  and $W^{T\R^n}$.)}
\end{itemize}
The existence of solutions is achieved via approximate
solutions called Lattice Approximate Solutions (briefly LAS).
To define in details LAS, we need some additional notation.

For $N\in\N$ let $\Delta_N =\frac{1}{N}$ be the time step size,
$\Delta^v_N=\frac{1}{N}$ the velocity step size and
$\Delta^x_N=\Delta^v_N\Delta_N=\frac{1}{N^2}$ the space step size.
The discretization in position and velocity are given by the
points $x_i$ of $\Z^n/(N^2)\cap [-N,N]^n$
and $v_j$ of $\Z^n/N\cap [-N,N]^n$.
Every $\mu\in\M(\R^n)$ can be approximated by Dirac deltas using
the operator:
\begin{equation}\label{eq:disc-x}
\A^x_N(\mu)=\sum_i m^x_i(\mu) \delta_{x_i}
\end{equation}
with
%\begin{equation}\label{eq:def:mx}
$m^x_i(\mu)=\mu(x_i+Q)$
%\end{equation}
amd  $Q=([0,\frac{1}{N^2}[)^n$.
The PVF can also be approximated using the operator:
\begin{equation}\label{eq:disc-v}
\A^v_N(V[\mu])= \sum_i
\sum_j m^v_{ij}(V[\mu])\ \delta_{(x_i,v_j)}
\end{equation}
where
%\begin{equation}\label{eq:def:mv}
$m^v_{ij}(V[\mu])=V[\mu](\{(x_i,v):v\in v_j+Q'\})$,
%\end{equation}
and $Q'=([0,\frac{1}{N}[)^n$.

The two operators are good approximations for the
Wasserstein distances in the following sense:
\begin{lemma}\label{lem:As}
Given $\mu\in\PP_c(\R^n)$, for $N$ sufficiently big the following holds:
\[
W(\A^x_N(\mu),\mu)\leq  {\sqrt{n}}\, \Delta^x_N,\qquad
W^{T\R^n}(\A^v_N(V[\mu]),V[\mu])\leq  {\sqrt{n}}\, \Delta^v_N.
\]
\end{lemma}
The definition of LAS is as follows.
\begin{definition}\label{def:LAS}
Consider $V$ satisfying  {(H:bound)}, (\ref{eq:MDE-Cauchy}), $T>0$ and $N\in\N$.
% {such that $e^{C_N T} (R_N+1)<N$ (see Lemma \ref{lem:bound-supp} for definition of $C_N$ and $R_N$)},
%we define the Lattice Approximate Solution (LAS)
The LAS $\mu^N:[0,T]\to \PP_c(\R^n)$ is defined by recursion.
%Recalling (\ref{eq:disc-x})-(\ref{eq:def:mv}),
%we set
First $\mu_0^N=\A^x_N(\mu_0)$ and then:
\begin{equation}\label{eq:def-rec}
\mu^N_{\ell+1}=\mu^N((\ell+1)\Delta_N)=\sum_i \sum_j
m^v_{ij}(V[\mu^N(\ell\Delta_N)])\ \delta_{x_i+\Delta_N\, v_j}.
\end{equation}
%By definition of $\Delta_N$, $\Delta^v_N$, $\Delta^x_N$ and (\ref{eq:def-rec}),
Notice that  $Supp(\mu^N_\ell)$ is  contained in the set $\Z^n/(N^2)\cap [-N,N]^n$, thus
$\mu^N_\ell=\sum_i m^{N,\ell}_i \delta_{x_i}$
for some $m^{N,\ell}_i\geq 0$.
We can define LAS for all times by interpolation:
\begin{equation}\label{eq:def-mu-int}
\mu^N(\ell\Delta_N+t)=\sum_{ij}
m^v_{ij}(V[\mu^N(\ell\Delta_N)])\ \delta_{x_i+t\, v_j}.
\end{equation}
\end{definition}
It is easy to prove uniform bounds for LAS (\cite{Piccoli-ARMA}):
\begin{lemma} \label{lem:bound-supp}
Given a PVF $V$ satisfying (H:bound),  $\mu_0$ with $Supp(\mu_0)\subset B(0,R)$
and $\ell$ such that $\ell\Delta_N\leq T$, the following holds true:
\begin{equation}
Supp(\mu^N_\ell) \subset
%B\left(0,e^{C_N\ell\Delta_N}R+e^{C_N\ell\Delta_N}-1\right) \subset
B\left(0,e^{C_N T} (R_N+1)-1\right),
\end{equation}
where $C_N=C+\frac{\sqrt{n}}{N}$ and $R_N=R+\frac{\sqrt{n}}{N^2}$.
\end{lemma}

\subsection{Finite speed diffusion via MDEs}
MDEs can be used to model finite speed diffusion.
Let us first illustrate a simple example of mass splitting,
which is also related to the Wasserstein gradient flow with
interaction energy $\Phi(\mu):=-\int_{\R\times\R} |x-y| d\mu(x)d\mu(y)$, see \cite{BCDFP15}.
\begin{example}\label{ex:1}
Given $\mu\in\M(\R)$ we denote by $B(\mu)$ the barycenter of $\mu$
(i.e. $B(\mu)=\sup \left\{x:\mu(]-\infty,x])\leq \frac12\right\}$)
and set $V[\mu]=\mu\otimes \nu_x$, with:
\begin{equation}\label{eq:PVF-Bar}
\nu_x=\left\{
\begin{array}{ll}
\delta_{-1} & \textrm{if}\ x<B(\mu)\\
\delta_{1} & \textrm{if}\ x>B(\mu)\\
 {\frac{1}{\mu(\{B(\mu)\})}
\left(\eta\delta_{1}+
\left(\frac12-\mu(]-\infty,B(\mu)[)\right)\delta_{-1}\right)} & \textrm{if}\ x=B(\mu),
\mu(\{B(\mu)\})>0
\end{array}
\right.
\end{equation}
where $\eta=\mu(]-\infty,B(\mu)]) - \frac12$.
\end{example}
We have the following:
\begin{prop}\label{prop:Ex1}
The LAS for the PVF (\ref{eq:PVF-Bar}) converges to:
\[
\mu(t)(A) = \mu_0((A\cap ]-\infty,B(\mu)-t[)+t)
+ \mu_0((A\cap ]B(\mu)+t,+\infty[)-t)
\]
\[
+ {\frac{1}{\mu_0(\{B(\mu_0)\})}
\left( \eta \delta_{B(\mu_0)+t}(A)+ (\frac12-\mu_0(]-\infty,B(\mu_0)[))\delta_{B(\mu_0)-t}(A)\right).}
\]
\end{prop}
Notice that from Proposition \ref{prop:Ex1} we deduce that
for the initial datum $\mu_0=\delta_{x_0}$ the solution
split mass in half, i.e.
$\mu(t)=\frac12 \delta_{x_0+t}+\frac12 \delta_{x_0-t}$.

Example \ref{ex:1} can be generalized as follows.
\begin{example}\label{ex:2}
Consider an increasing map $\varphi:[0,1]\to\R$ and define
$V_\varphi[\mu]=\mu\otimes_x J_\varphi(x)$, where
\[
J_\varphi(x)=\begin{cases}
\delta_{\varphi(F_\mu(x))}&\mbox{~~if~}F_\mu(x^-)=F_\mu(x),\\
\frac{\varphi\#\left(\chi_{[F_\mu(x^-),F_\mu(x)]}\lambda\right)}{F_\mu(x)-F_\mu(x^-)}&\mbox{~~otherwise,}
\end{cases}
\]
where
$F_\mu(x)=\mu(]-\infty,x])$, the cumulative distribution of $\mu$,
and $\lambda$ the Lebesgue measure. In simple words $V[\mu]$
moves the ordered masses with speed prescribed by $\varphi$.\\
If $\varphi$ is a diffeomorphism, the solution from $\delta_0$
is given by $g(t,x)\lambda$ with
$$g(t,x)=\frac{1}{t\varphi '(\varphi^{-1}(\frac{x}{t}))}=
\frac{(\varphi^{-1})'(\frac{x}{t})}{t}.$$
For $\varphi(\alpha)=\alpha-\frac12$ we get
$g(t,x)=\frac{1}{t}\chi_{[-\frac{t}{2},\frac{t}{2}]}$.
In general $V_\varphi$ gives rise to any $g$, which is solution to the equation $g_t+\frac{x}{t}g_x=0$.
\end{example}

\section{General theory for coupled ODE-MDE}
In this section we provide existence and uniqueness results
for systems consisting of a MDE coupled with an ODE.

\begin{definition}
A coupled ODE-MDE system is a system written as:
\begin{equation}\label{eq:ODE-MDE}
\left\{
\begin{array}{c}
\dot{x}=g(x,\mu)      \\
\dot{\mu}=V[\mu]+s(\mu,x)
\end{array}
    \right.
\end{equation}
where $g:\R^m\times\M(\R^n)\to \R^m$, $V:\M(\R^n)\to \M(T\R^n)$ is a PVF
on $\M(\R^n)$ and $s:\M(\R^n)\times \R^m\to \M(\R^n)$.
\end{definition}

\begin{definition}\label{def:sol-ODE-MDE}
A solution to \eqref{eq:ODE-MDE} with given initial datum $(x_0, \mu_0) \in \mathbb{R}^m \times \mathcal{M}(\mathbb{R}^n)$ is a couple $(x,\mu)$,
with $x:[0,T]\to \R^m$ and  $\mu:[0,T]\to \mathcal{M}(\R^n)$
such that for each $f\in{\C}^\infty_c(\R^n)$ the following holds:
\begin{itemize}
\item[i)]\hspace*{-6pt} {\small The map $t\to x(t)$ is absolutely continuous and satisfies
$x(t)=x_0+\int_0^t g(x(\tau),\mu(\tau))\,d\tau$} for almost every $t\in [0,T]$.
\item[ii)] $\mu(0) = \mu_0$;
\item[iii)] There exists $C(T)>0$ such that $|\mu(t)|\leq C(T)$;
    \item[iv)] The integral
$\int_{T\R^n} (\nabla f(y)\cdot v)\, dV[\mu(t)](y,v)$ is defined for almost every $t \in [0, T]$;
 \item[v)] The map $t\to \int_{\R^n} f(y)\, ds[\mu(t),x(t)](y)$ belongs to $L^1([0,T])$;
   \item[vi)] The map $t\to \int_{\R^n} f(y)\, d\mu(t)(y)$ is absolutely
continuous and for almost every $t\in [0,T]$ it satisfies:
\begin{equation}\label{eq:sol-MDE-w-s}
\frac{d}{dt}\int_{\R^n} f(y)\,d\mu(t)(y) =
\int_{T\R^n} (\nabla f(y)\cdot v)\ dV[\mu(t)](y,v)+\int_{\R^n} f(y) \,ds[\mu(t),x(t)](y).
\end{equation}
\end{itemize}
\end{definition}

\begin{definition}\label{def:Lip-semigroup-OM}
A Lipschitz semigroup for \eqref{eq:ODE-MDE}
is a map $S:[0,T]\times \R^m\times \M(\R^n) \to \R^m\times \M(\R^n)$ such that
for every $(x,\mu), (y,\nu) \in \R^m\times \M(\R^n)$ and $s,t\in [0,T]$ the following holds:
\begin{itemize}
\item[i)] $S_0(x,\mu)=(x,\mu)$ and $S_t\,S_s\,(x,\mu)=S_{t+s}\,(x,\mu)$;
\item[ii)] The map $t\mapsto S_t(x,\mu)$ is a solution to \eqref{eq:ODE-MDE} in the sense of Definition~\ref{def:sol-ODE-MDE};
\item[iii)] Denote by $S^i_t$, $i=1,2$, the components of $S_t$,
so that $S_t=(S_t^1, S_t^2)$ with
$S^1_t \colon \R^m \times \mathcal{M}(\R^n) \to \mathbb{R}^m$ and $S^2_t \colon \R^m \times \mathcal{M}(\R^n) \to \mathcal{M}(\mathbb{R}^n)$.
%for every $t \in [0, T]$, let $S_t^1=\text{proj}_1(S_t)$ and $\text{proj}_2(S_t)=S_t^2$ where $\text{proj}_1$ and $\text{proj}_2$ are the projection maps $\text{proj}_1(S_t) \colon \R^m \times \mathcal{M}(\R^n) \to \mathbb{R}^m$ and $\text{proj}_2(S_t) \colon \R^m \times \mathcal{M}(\R^n) \to \mathcal{M}(\mathbb{R}^n)$ such that $S_t=(\text{proj}_1(S_t), \text{proj}_2(S_t)) =(S_t^1, S_t^2) $.
For every $R, M>0$ there exists $C=C(R, M)>0$ such that
if $|x|,|y|\leq M$
$Supp(\mu) \cup Supp(\nu) \subset B(0,R)$ and $\mu(\R^n) + \nu(\R^n) \leq M$ then we have for every $s,t\in [0,T]$:
\begin{equation}\label{eq:ODE-bound-OM}
|S^1_t(x,\mu)|\leq C;
\end{equation}
\begin{equation}\label{eq:supp-bound-OM}
Supp(S^2_t(x,\mu))\subset B(0,e^{Ct} (R+M+1));
\end{equation}
\begin{equation}\label{eq:cont-dep-OM}
|S^1_t(x,\mu)-S^1_t(y,\nu)|+W^g(S^2_t(x,\mu),S^2_t(y,\nu))\leq e^{Ct}\left(|x-y|+W^g(\mu,\nu)\right);
\end{equation}
\begin{equation}\label{eq:Lip-time-cont-OM}
|S^1_t(x,\mu)-S^1_s(x,\mu)|+W^g(S^2_t(x,\mu),S^2_s(x,\mu))
\leq C\ |t-s|.
\end{equation}
\end{itemize}
\end{definition}

We first state our assumptions on \eqref{eq:ODE-MDE} for existence
of a semigroup of solutions.

\begin{theorem}[Existence of Lipschitz semigroup of solutions to \eqref{eq:ODE-MDE}]\label{th:ODE-MDE}
Consider the ODE-MDE system \eqref{eq:ODE-MDE} and assume
the following. The PVF $V$ satisfies the assumption (H:bound) and:
\begin{description}
%\begin{itemize}
\MyItem[(OM:Lip-g)] for each $R>0$ there exists $L=L(R)>0$ such that if
$|x|,|y|\leq R$ and $\mathrm{supp}(\mu)\cup\mathrm{supp}(\nu)\subset B(0,R)$ then:
\begin{equation}
|g(x,\mu)-g(y,\nu)|\leq L\ \left(|x-y|+ W^g(\mu, \nu)\right);
\end{equation}
\MyItem[(OM:Lip-V)] for each $R>0$ there exists $K=K(R)>0$ such that if
$\mathrm{supp}(\mu)\cup\mathrm{supp}(\nu)\subset B(0,R)$ then:
\begin{equation}
\mathcal{W}^g(V[\mu],V[\nu])\leq K W^g(\mu, \nu);
\end{equation}
\MyItem[(OM:bound-s)] there exists $\bar{R}$ such that for all $x\in\R^m$ and
$\mu\in \M(\R^n)$ it holds $\mathrm{supp}(s[\mu, x])\subset B(0, \bar R)$
and $|s[\mu, x]|\leq \bar{R}$.
\MyItem[(OM:Lip-s)] there exists $M$ such that for all $x,y\in\R^m$ and
$\mu,\nu\in \M(\R^n)$ it holds
\begin{equation}
W^g(s[\mu, x],s[\nu, y])\leq M \left(|x-y|+ W^g(\mu, \nu)\right).
\end{equation}
%\end{itemize}
\end{description}
Then, there exists a Lipschitz semigroup of solutions to \eqref{eq:ODE-MDE}.
\end{theorem}

To prove Theorem \ref{th:ODE-MDE}, we define an Euler-LAS scheme to construct approximate solutions. The idea is to use the standard Euler explicit scheme for the ODE and the LAS
for the MDE. The details are as follows.
\begin{definition}\label{def:Euler-LAS}
An Euler-LAS approximate solution to \eqref{eq:ODE-MDE} is obtained
as follows. For fixed $N$, set $\Delta_N=\frac{1}{N}$ and perform the following steps.
\begin{itemize}
    \item[Step 1] Initial step: Define $x_0^N=x_0$ and $\mu_0^N=\A^x_N(\mu_0) = \sum \limits_{i} m_i^x(\mu_0) \delta_{x_i}$;
    \item[Step 2] Inductive Step: Define
    $$\mu^N_{\ell+1}= \sum_i \sum_j m^v_{ij}(V[\mu^N_\ell])\ \delta_{x_i+\Delta_N\, v_j} +\Delta_N \A^x_N(s[\mu^N_\ell, x^N_\ell]),$$
    $$x^N_{\ell+1}=x^N_\ell+\Delta_N g(x^N_\ell,\mu^N_{\ell});$$
    \item[Step 3] Interpolated measure:
    For the intermediate times $\tau \in (0, \Delta_N)$,
    \begin{equation}
    \label{approximated_solution_measure}
    \mu^N_{\ell}(\tau)= \sum_i \sum_j m^v_{ij}(V[\mu^N_\ell])\ \delta_{x_i+\tau\, v_j} +\tau \A^x_N(s[\mu^N_\ell, x^N_\ell]),\end{equation}
    \begin{equation}
    \label{approximated_solution_space}
    x^N_{\ell}(\tau)=x^N_\ell+\tau g(x^N_\ell,\mu^N_{\ell}).\end{equation}
    \end{itemize}
\end{definition}

\begin{lemma}[Uniform boundedness]
\label{lem_bound}
Under the assumption in Theorem \ref{th:ODE-MDE}, for every $N, \ell \in \mathbb{N}^{+}$ the measure $\mu^N_\ell$ has uniformly bounded mass and support and $x^N_\ell$ is uniformly bounded in norm, where $\mu^N_\ell$ and $x_\ell^N$ are  defined as in Definition \ref{def:Euler-LAS}.
\end{lemma}

\begin{proof}
First fix a time $T>0$ and initial conditions $(x_0,\mu_0) \in \mathbb{R}^m \times \mathcal{M}(\R^n)$ for the system \eqref{eq:ODE-MDE}. Choose $\bar{R}>0$ in {(OM:bound-s)} being an upper bound of the maximal support and the total mass of $s[x, \mu]$ for all $x \in \mathbb{R}^m$ and $\mu \in \mathcal{M}(\mathbb{R}^n)$. One can also enlarging $\bar{R}$ such that $\mathrm{supp}(\mu_0)\subset B(0, \bar{R})$. Let $(x^N_\ell,\mu^N_\ell)$ be the sequence
of Euler-LAS approximate solutions defined in Definition \ref{def:Euler-LAS}, $N \in \mathbb{N}^{+}$.

First notice that $\mathrm{supp}(\mu_0) \subset B(0, \bar{R})$ implies that $\mathrm{supp}(\mu^N_0) \subset B(0 ,\bar{R}+1)$. In addition, for all $x \in \R^m, \mu \in \mathcal{M}(\R^n)$, $\mathrm{supp}(s[\mu, x]) \subset B(0, \bar{R})$ implies that $\mathrm{supp}(\mathcal{A}_N^x(s[\mu, x])) \subset B(0, \bar{R}+1)$.

Furthermore, from (H:bound) we get:
for each term $i, j$ in equation \eqref{approximated_solution_measure} it holds that $(x_i, v_j) \in \mathrm{supp}(V[\mu_\ell^N])$ implies that $|v_j| \leq C(1+ \sup\limits_{x\in Supp(\mu^N_\ell)}|x| )$. Therefore,
\[
\sup_{x\in Supp(\mu^N_{\ell}(\tau))} |x|\leq
\max\{\bar{R},\sup_{x\in Supp(\mu^N_{\ell})}|x|\}
+ \Delta_N \cdot C \left(1+\sup_{x\in Supp(\mu^N_\ell)}|x|\right),
\]
Note that $\ell \leq \frac{T}{\Delta N}+1$ and by induction one can derive that $Supp(\mu^N_\ell)$ is uniformly bounded by
\[
C_1=e^{CT}\max\{\bar{R},\sup_{x\in Supp(\mu_0)}|x|\}.
\]
Now observe that:
\begin{align}
\label{x_bound_1}
  |\mu^N_{\ell+1}|\leq |\mu^N_\ell|+\Delta_N \bar{R},
\end{align}
\begin{align}
\label{x_bound_3}
  W^g(\mu^N_{\ell+1},\mu_0)& \leq  W^g(\mu^N_{\ell},\mu_0) + W^g(\mu^N_{\ell+1},\mu_\ell^N)\\ \nonumber
  & \leq  W^g(\mu^N_{\ell},\mu_0)+\Delta_N(1+C_1) |\mu^N_\ell|+\Delta_N \bar{R},
\end{align}
and
\begin{align}
\label{x_bound_2}
  |x^N_{\ell+1}-x^N_\ell| &= \left|\Delta_N g(x^N_\ell,\mu^N_{\ell})\right|\\ \nonumber
  &\leq \Delta_N\left(|g(x_0, \mu_{\ell}^N)|+L|x_\ell^N-x_0|\right)\\ \nonumber
  & \leq \Delta_N\left(|g(x_0, \mu_0)|+W^g(\mu_\ell^N, \mu_0)+L|x_\ell^N-x_0|\right).
\end{align}

From inequality \eqref{x_bound_1} we get that the mass of $\mu^N_\ell$
is uniformly bounded. This implies, using inequality  \eqref{x_bound_3}
that $W^g(\mu^N_\ell,\mu_0)$ is also bounded. Finally, inequality  \eqref{x_bound_2} implies that $x^N_\ell$ is bounded in $\R^m$.
\end{proof}
\begin{lemma}[Existence of weak solution to system \eqref{eq:ODE-MDE}]
Under the assumption in Theorem \ref{th:ODE-MDE}, there exists a limit curve $(x, \mu) \colon [0, T] \to \mathbb{R}^{m} \times \mathcal{M}(\mathbb{R}^n)$ such that $\mu^N  \rightharpoonup \mu$ and $x^N \to x $ as $N \to \infty$, where $\mu^N$ and $x^N$ are  defined as in Definition \ref{def:Euler-LAS}. Moreover, the limit curve $(x, \mu)$ is a weak solution to system \eqref{eq:ODE-MDE} as defined
 Definition \ref{def:sol-ODE-MDE}.
\end{lemma}
\begin{proof}
 By Lemma \ref{lem_bound}, we have that $\mu^N_\ell$ is pre-compact in $\M(\R^n)$
and $x^N_\ell$ is pre-compact in $\R^m$. By a diagonal argument,
we can define a limit curve $(x,\mu)(t)$ defined on $[0,T]$. We claim that the limit curve $(x,\mu) \colon [0, T] \to \mathbb{R}^m \times \mathcal{M}(\R^n)$ is  a weak solution to system \eqref{eq:ODE-MDE} as defined
 Definition \ref{def:sol-ODE-MDE}.

First observe that $W^g(\mu_0^N, \mu_0) \leq \Delta_N^x|\mu_0| \to 0 $ as $N \to \infty$. Thus $\mu(0) = \mu_0$. By the definition of $x_0^N$, it is clear that $x(0)=x_0$. By the definition of $x^N$ and the fact that $\mu^N  \rightharpoonup \mu$ and $x^N \to x $ as $N \to \infty$, we have, $x(t) = x_0 + \int_0^t g(x(\tau), \mu(\tau))\, d\tau$ for a.e. $t \in [0, T]$.

For fixed $f \in \mathcal{C}_c^{\infty}(\R^n)$, and $\sigma, \tau \in [0, T]$ with $\sigma > \tau$, define the operator $F^N$ as follows:
\begin{align*}
    &F^N(\sigma, \tau) \coloneqq \int_{\mathbb{R}^n} f(y)\, d(\mu^N(\sigma)-\mu^N(\tau))(y)\\
    &-\int_\tau^\sigma \, dt \left( \int_{T\R^n} (\nabla f(y)\cdot v)\ dV[\mu^N(t)](y,v)+\int_{\R^n} f(y) \,ds[\mu^N(t),x^N(t)](y)\right).
\end{align*}

Definition \ref{def:Euler-LAS} implies that for $\sigma, \tau \in [\ell \Delta_N, (\ell +1)\Delta_N]$,
\begin{align*}
    &\int_{\mathbb{R}^n} f(y)\, d(\mu^N(\sigma)-\mu^N(\tau))(y) = \\
    =& \int_{\mathbb{R}^n} f(y)\, \left(\sum\limits_{i,j} m_{ij}^v(V[\mu_\ell^N]) \, d(\delta_{x_i+(\sigma-\ell \Delta_N)v_j}-\delta_{x_i + (\tau-\ell \Delta_N)v_j})(y)+ \right.\\
    &+\left.(\sigma-\tau) d\mathcal{A}_N^{x}(s[\mu_\ell^N, x_\ell^N])(y)\right)
\end{align*}
Define $g_{ij}(\alpha) \coloneqq f(x_i+\alpha(\sigma -\ell \Delta_N)v_j) - f(x_i+\alpha(\tau-\ell\Delta_N)v_j)$. Then $g_{ij}(1)= \int_{\mathbb{R}^n} f(y)\, d(\delta_{x_i+(\sigma-\ell \Delta_N)v_j}-\delta_{x_i + (\tau-\ell \Delta_N)v_j})(y)$. In addition, one can bound $|F^N(\sigma, \tau)|$ from above by
\begin{align}
\label{FN_upp_bound}
    &\left|\sum\limits_{i,j} g_{ij}(1) m_{ij}^v(V[\mu_\ell^N])-\int_\tau^\sigma \, dt \int_{T\R^n} (\nabla f(y)\cdot v)\ dV[\mu_\ell^N(t)](y,v) \right|+\\ \nonumber
    &\left|(\sigma-\tau)\int_{\R^n} f(y) d\mathcal{A}_N^{x}(s[\mu_\ell^N, x_\ell^N])(y)-\int_{\tau}^\sigma \, dt \int_{\R^n} f(y) \,ds[\mu_\ell^N(t),x_\ell^N(t)](y)\right|\\ \nonumber
    \leq & |\tau-\sigma|\Delta_N\|f\|_{C^3}\tilde{C}
\end{align}
for a suitable constant $\tilde{C}$. For more details, please see \cite{PR19}.

For a general pair $\sigma, \tau \in [0, T]$ with $\tau < \sigma$, assume that $N$ is sufficiently large such that
\[(k_1-1)\Delta_N \leq \tau < k_1\Delta_N \leq k_2 \Delta_N < \sigma \leq (k_2+1)\Delta_N\] for some $k_1, k_2 \in \mathbb{N}^{+}\setminus \{0\}$. Consider the convergent subsequence $\mu^N \rightharpoonup \mu$, and define
\begin{align*}
    \mathcal{L} = &\lim \limits_{N \to \infty} \frac{1}{\sigma - \tau} F^N(\sigma, \tau)
    =\frac{1}{\sigma - \tau}\left( \int f(y)\, d(\mu(\sigma)-\mu(\tau))\right.\\
    &\left.-\int_\tau^\sigma \, dt \int_{T\R^n} (\nabla f(y)\cdot v)\ dV[\mu(t)](y,v)-\int_\tau^{\sigma}\int_{\R^n} f(y) \,ds[\mu(t),x(t)](y)\right).
\end{align*}
The second identity follows from  the continuity of both $V$ and $s$ with respect to weak convergence of measures.
We claim that $\lim\limits_{\sigma \to \tau} |\mathcal{L}|=0$.
In fact, we have
\begin{equation}
\label{eqn: FN}
F^N(\sigma, \tau) = F^N(\tau, k_1\Delta_N) + \sum \limits_{k=k_1}^{k_2-1} F^N(k \Delta_N, (k+1)\Delta_N) + F^N(k_2\Delta_N, \sigma).
\end{equation}
Combining equations \eqref{FN_upp_bound} and \eqref{eqn: FN}, we get
\begin{align*}
    \lim\limits_{\sigma \to \tau} |\mathcal{L}| &=\lim \limits_{\sigma \to \tau} \left|\lim \limits_{N \to \infty} \frac{1}{\sigma - \tau} F^N(\sigma, \tau)\right|\\
    &= \lim\limits_{\sigma \to \tau} \frac{1}{\tau-\sigma} \lim \limits_{N \to \infty} |F^N(\sigma, \tau)| \\
    & \leq \lim\limits_{\sigma \to \tau} \frac{1}{\tau-\sigma}  \lim \limits_{N \to \infty} (|\tau-k_1\Delta_N|+\dots+|k_2\Delta_N - \sigma|)\Delta_N\|f\|_{C^3}\tilde{C}\\
    &=0.
\end{align*}
\end{proof}

Now we are ready to prove Theorem \ref{th:ODE-MDE}.

\begin{proof}
We only need to verify equations \eqref{eq:cont-dep-OM} and \eqref{eq:Lip-time-cont-OM} in Definition \ref{def:Lip-semigroup-OM}.
Choose two different sets of initial datum $(x_0, \mu_0)$ and $(y_0, \nu_0)$ in $\mathbb{R}^m \times \mathcal{M}(\mathbb{R}^n)$, and build Euler-LAS approximate solutions $(x^N, \mu^N)$ and $(y^N, \nu^N)$ to \eqref{eq:ODE-MDE} according to Definition \ref{def:Euler-LAS}.

Thanks to the uniform boundedness of the supports of $\mu^N$, $\nu^N$, $V[\mu^N]$, $V[\nu^N]$ and the fact that $ \mathcal{A}_N^{x}(\mu_{0}^N) = \mu^N_0$, we can assume that $N$ is sufficiently large such that for all $\ell \in \mathbb{N}$ with $\ell \leq \frac{T}{\Delta_N}$the followings are true:

\begin{align*}
    \mathcal{A}_N^{x}(\mu_{\ell}^N) &= \mu^N_\ell\\
     \mathcal{A}_N^{x}(\nu_{\ell}^N) &= \nu^N_\ell\\
     \pi_1 \# \mathcal{A}_N^v(V[\mu_\ell^N]) &= \pi_1 \#V[\mu_\ell^N]\\
     \pi_1 \# \mathcal{A}_N^v(V[\nu_\ell^N]) &= \pi_1 \#V[\nu_\ell^N].\\
\end{align*}

Let $V_1 \coloneqq \mathcal{A}_N^v(V[\mu^N_\ell])$ and $V_2 \coloneqq \mathcal{A}_N^v(V[\nu^N_\ell])$. By definitions of $\mu^N$, $\nu^N$, the generalized Wasserstein distance $W^g$ and the operator $\mathcal{W}^g$, one can estimate recursively as
\begin{equation}
\label{eqn: W_g_W_g_1}
W^g(\mu^N_{\ell+1}, \nu^N_{\ell+1}) \leq W^g(\mu^N_{\ell}, \nu^N_{\ell}) + \Delta_N \mathcal{W}^g(V_1, V_2).
\end{equation}

Furthermore, we can compare $\mathcal{W}^g(V_1, V_2)$ with $ \mathcal{W}^g(V[\mu^N_{\ell}], V[\nu^N_{\ell})]$ as

\begin{equation}
\label{eqn: W_g_W_g_2}
    \mathcal{W}^g(V_1, V_2) \leq \mathcal{W}^g(V[\mu^N_{\ell}], V[\nu^N_{\ell})]+2\sqrt{n}\Delta_N|\mu^N_\ell|.
\end{equation}
For more details, please see \cite{PR19}.

Now combine equations \eqref{eqn: W_g_W_g_1} and \eqref{eqn: W_g_W_g_2}, and use the hypothesis (OM:Lip-V) and Lemma \ref{lem_bound}, we have, there exists $\bar{C}>0$, such that
\begin{equation}
    \label{eqn: W_g_W_g_3}
    W^g(\mu^N_{\ell+1}, \nu^N_{\ell+1}) \leq (1+K) W^g(\mu^N_{\ell}, \nu^N_{\ell}) + 2 \sqrt{n} (\Delta_N)^2 \bar{C}.
\end{equation}
By induction and Lemma \ref{lem:As}, equation \eqref{eqn: W_g_W_g_3} implies that
\begin{equation}
    \label{eqn: W_g_W_g_4}
    W^g(\mu^N(t), \nu^N(t)) \leq e^{Kt}W^g(\mu_0, \nu_0)+2 \sqrt{n}(\Delta_N)^2\bar{C}\frac{e^{Kt}-1}{K}.
\end{equation}

Now we pass to the limit by using a density argument. Let
\[\mathcal{D} = \left\{\mu_0 \in \mathcal{M}, \text{ s.t. } \mu_0 = \sum \limits_i m_i \delta_{x_i}, m_i \in \mathbb{Q}^{+}, x_i \in \mathbb{Q}^n \right\}.\]
Choose $(x_0, \mu_0), (y_0, \nu_0) \in \mathbb{R}^m \times \mathcal{D}$, define the corresponding sequences $(x^N, \mu^N)$ and $(y^N,\nu^N)$ according to Definition \ref{def:Euler-LAS} such that there exist subsequences $(x^{N_k}, \mu^{N_k})$ and $(y^{N_k}, \nu^{N_k})$ converges to solutions $(x, \mu)$ and $(y, \nu)$ to system \eqref{eq:ODE-MDE}. Repeat this diagonal argument for initial data in $\mathcal{D}$ and pass to the limit in inequality \eqref{eqn: W_g_W_g_4} for $N \to \infty$, we have
\[W^g(\mu(t), \nu(t)) \leq e^{Kt}W^g(\mu_0, \nu_0) \text{ for all } \mu_0, \nu_0 \in \mathcal{D}.\]
Since the set $\mathcal{D}$ is countable and dense in $\mathcal{M}$, the Lipschitz continuity with respect to the initial condition can be extended to $\mathcal{M}$.
Finally, using  hypothesis (OM:Lip-g), we obtain \eqref{eq:cont-dep-OM}.

To prove equation \eqref{eq:Lip-time-cont-OM}, the uniform Lipschitz continuity of the solution to system \eqref{eq:ODE-MDE} with respect to time, it is enough to notice that the sequence $\mu^{N}$ are equi-Lipschitz in time with respect to the distance $W^g$. For more details we refer the reader to \cite{PR19}.
\end{proof}

Next we will provide the uniqueness result for solutions to system \eqref{eq:ODE-MDE}. Due to the weak concept of solution, uniqueness
can be achieved only at the level of semigroup by prescribing small-time
evolution of finite sums of Dirac masses as in \cite{Piccoli-ARMA}.
First of all, we recall the definitions of Dirac germs compatible with a given PVF  and semigroup compatible with a given Dirac germ.
\begin{definition}[Dirac germ compatible with a PVF]
\label{def: Dirac_germ}

Let $V$ be a given fixed PVF and define $\mathcal{M}^D \coloneqq \{\mu \in \mathcal{M}(\mathbb{R}^n) \text{ such that } \mu=\sum\limits_{l=1}^{m} m_l \delta_{x_l},  m_l \in \mathbb{Q}^{+}, x_l \in \mathbb{Q}^{n}\}$ the space of discrete measures. A Dirac germ $\gamma$ compatible with $V$ is a map that assigns to each $\mu \in \mathcal{M}^{D}$ a Lipschitz curve $\gamma_{\mu} \colon [0, \varepsilon(\mu)] \to \mathcal{M}(\mathbb{R}^n)$ such that
\begin{itemize}
    \item[(i)] $\varepsilon[\mu]>0$ is uniformly positive for measures with uniformly bounded~ support.
    \item[(ii)] $\gamma_\mu$ is a solution to equation \eqref{eq:MDE} as defined in Definition \ref{def:sol-MDE}.
\end{itemize}
\end{definition}

\begin{definition}[Compatibility of a semigroup for system \eqref{eq:ODE-MDE}]
\label{def: Compatinility_semigroup}
Fix a given PVF $V$ that satisfies (H: bound), a final time $T>0$, and a Dirac germ $\gamma$ that is compatible with $V$ as in Definition \ref{def: Dirac_germ}. We say that a semigroup $S \colon [0, T] \times \mathbb{R}^m \times \mathcal{M}(\mathbb{R}^n) \to \mathbb{R}^m \times \mathcal{M}(\mathbb{R}^n)$ for system \eqref{eq:ODE-MDE} is compatible with $\gamma$ if for each $R, M>0$ there exists $C(R,M)$ such that for fixed $x_0 \in \mathbb{R}^m$ with $|x_0|\leq M$, the space $\mathcal{M}_{R, M}^D \coloneqq \{\mu \in \mathcal{M}^D \text{ such that } Supp(\mu) \in B(0, R), |\mu| \leq M\}$ satisfies for all $t \in [0, \inf\limits_{\mu \in \mathcal{M}_{R, M}^D}\varepsilon(\mu)], x \in \mathbb{R}^m$, one has,
\begin{equation}
\label{eqn: compatibility_semigroup}
W^g(S_t^2(x_0,\mu), \gamma_{\mu}(t)) \leq C(R,M)t^2.
\end{equation}
\end{definition}

The following lemma plays an essential rule to prove the uniqueness of solutions to system \eqref{eq:ODE-MDE}.
The proof is entirely similar to that provided in \cite{Piccoli-ARMA},
thus we skip it.
\begin{lemma}
\label{lem: uniqueness}
Let $S\colon [0, T] \times \mathbb{R}^m \times \mathcal{M}(\mathbb{R}^n) \to \mathbb{R}^m \times \mathcal{M}(\mathbb{R}^n)$ be a Lipschitz semigroup for system \eqref{eq:ODE-MDE}, $\mu \colon [0, T] \to \mathcal{M}(\mathbb{R}^n)$ a curve that is Lipschitz continuous, and $x_0 \in \mathbb{R}^m$. Then it holds
\[W^g(S_t^2(x_0, \mu(0)), \mu(t)) \leq e^{Ct} \int_0^t \liminf \limits_{h \to 0^{+}} W^g(S_h^2(x_0, \mu(s)), \mu(s+h)) \, ds\]
where $C$ is the Lipschitz constant in equation
\eqref{eq:cont-dep-OM}.
\end{lemma}

We are now ready to prove the uniqueness of a semigroup to system \eqref{eq:ODE-MDE} conmpatible with a Dirac germ.
\begin{theorem}[Uniqueness of Lipschitz semigroup of solutions to \eqref{eq:ODE-MDE}]
Under the same assumptions of Theorem \ref{th:ODE-MDE}, assume a Dirac germ $\gamma$ as in Defintion \ref{def: Dirac_germ} is given. Then there exists at most one Lipschitz semigroup for system \eqref{eq:ODE-MDE} that is compatible with $\gamma$ as in Defintion \ref{def: Compatinility_semigroup}.
\end{theorem}

\begin{proof}
Fix initial datum $(x_0, \mu_0) \in \mathbb{R}^m \times \mathcal{M}^{+}(\mathbb{R}^n)$ and $T>0$. Assume that there exist two semigroups $S, \bar{S}$ for system \eqref{eq:ODE-MDE} that are compatible with the given germ $\gamma$. By equation \eqref{eq:supp-bound-OM}, there exists $R>0$, such that $Supp(S_t^2(x_0, \mu_0)) \cup Supp(\bar{S}_t^2(x_0, \mu_0) \subset B(0, R)$ for all $t \in [0, T]$. By Lemma \ref{lem: uniqueness}, we have,
\begin{align}
\label{ineqn: uniquness}
&W^g(S_t^2(x_0, \mu_0), \bar{S}_t^2(x_0, \mu_0)) \nonumber \\
\leq &e^{Ct} \int_0^t \liminf \limits_{h \to 0^{+}} W^g(S_h^2(x_0, \bar{S}_s^2(x_0, \mu_0)), \bar{S}_{s+h}^2(x_0, \mu_0)) \, ds.
\end{align}
Fix $s>0$ and let $\nu = \bar{S}_s^2(x_0, \mu_0)$. Observe that $\mathcal{M}^D$ is dense in $\mathcal{M}$ with respect to the topology induced by $W^g$. Thus for every $\varepsilon>0$, there exists $\bar{\nu} \in \mathcal{M}^D$ such that $W^g(\nu, \bar{\nu}) < \varepsilon$. Applying equation \eqref{eqn: compatibility_semigroup} to both $S_h^2$ and $\bar{S}_h^2$, we have, there exists a  constant $C_1(R, M)$ such that
\[W^g(S_h^2(x_0,\bar{\nu}), \gamma_{\bar{\nu}}(h)) \leq C_1(R,M)h^2, \quad W^g(\bar{S}_h^2(x_0,\bar{\nu}), \gamma_{\bar{\nu}}(h)) \leq C_1(R,M)h^2.\]
Therefore
\begin{align*}
    &W^g(S_h^2(x_0, \bar{S}_s^2(x_0, \mu_0)), \bar{S}_{s+h}^2(x_0, \mu_0)) = W^g(S_h^2(x_0, \nu), \bar{S}_h^2(x_0, \nu)) \\
    \leq & W^g(S_h^2(x_0, \nu), S_h^2(x_0, \bar{\nu})) + W^g(S_h^2(x_0, \bar{\nu}), \gamma_{\bar{\nu}}(h)) \\
    & + W^g(\gamma_{\bar{\nu}}(h), \bar{S}_h^2(x_0, \bar{\nu}))+ W^g( \bar{S}_h^2(x_0, \bar{\nu}), \bar{S}_h^2(x_0, \nu))\\
    \leq & 2(e^{Ch}\varepsilon + C_1(R, M)h^2).
\end{align*}
Note that since both $s$ and $\varepsilon$ were chosen arbitrarily, for each $s>0$ it holds
\begin{equation}
\label{eqn: 2}
\liminf \limits_{h \to 0^{+}} W^g(S_h^2(x_0, \bar{S}_s^2(x_0, \mu_0)), \bar{S}_{s+h}^2(x_0, \mu_0))=0.
\end{equation}
Combining equations \eqref{ineqn: uniquness}
and \eqref{eqn: 2}, for every $t \in [0, T]$ it holds
\begin{equation}
\label{eqn: 5}
S_t^2(x_0, \mu_0)= \bar{S}_t^2(x_0, \mu_0).
\end{equation}
Furthermore, for a.e. $t \in [0, T]$, \[S_t^1(x_0, \mu_0) = x_0 + \int_0^t g(S_{\tau}^1(x_0, \mu_0),S_{\tau}^2(x_0, \mu_0)) \, d\tau,\]
and
\[\bar{S}_t^1(x_0, \mu_0) = x_0 + \int_0^t g(\bar{S}_{\tau}^1(x_0, \mu_0),\bar{S}_{\tau}^2(x_0, \mu_0)) \, d\tau.\]
Thus by assumption (OM:Lip-g) and equation \eqref{eqn: 5},
we conclude that for every $t \in [0, T]$,  $S_t^1(x_0, \mu_0)= \bar{S}_t^1(x_0, \mu_0)$.
\end{proof}

\section{Model for viral infections with mutations}

In this section we introduce a model for viral infections,
where mutations are occurring as viral dynamics in the infected
population.

We assume that the population
of infected is represented by a measure over a space of virus
mutations.
For simplicity we start assuming that the virus mutations
can be parameterized by one parameter $\alpha\in\R$ thus
$I\in\M(\R(\alpha))$.
On the other side, the population of susceptible can be
identified, as usual by a single parameter, thus $S\in\R$.
The dynamics of infections is captured by and ODE:
\[
\dot{S}=-\frac{S}{N} \int_{\R} \beta(\alpha) \ dI(\alpha),
\]
where $N$ is the total population containing the population of susceptible $S$, the population of infected $I$ and the population of recovered $R$,
$\beta(\alpha)$ is the infectivity rate, which depends on the virus mutation
identified by the parameter $\alpha$.

We now define a generalization of Example \ref{ex:2} to be applied
to the measure with time-varying mass $I(t)$.
First set $G_I(x)=\frac{I(]-\infty,x])}{I(\R)}$, which is the
normalized cumulative distribution, so that $G_I(-\infty)=0$ and
$G_I(+\infty)=1$.
Fix an increasing map $\varphi:[0,1]\to\R$ and define
$V_\varphi[I]=I\otimes_x J_\varphi(x)$, where
\begin{equation}\label{eq:MDE-cum-phi}
J_\varphi(x)=\begin{cases}
\delta_{\varphi(G_I(x))}&\mbox{~~if~}G_I(x^-)=G_I(x),\\
\frac{\varphi\#\left(\chi_{[G_I(x^-),G_I(x)]}\lambda\right)}{G_I(x)-G_I(x^-)}&\mbox{~~otherwise.}
\end{cases}
\end{equation}
The dynamics for $I$ is given by the MDE:
\[
\dot{I}=V_\varphi[I]+ \frac{S}{N} \beta(\alpha) I -\nu(\alpha) I
\]
where $\nu(\alpha)$ is the recovery rate, also dependent on the virus mutation. Finally a second ODE
described the dynamics of the population of recovered, $R$, by:
\[
\dot{R}=\int_{\R}\nu(\alpha) \ dI(\alpha).
\]

The overall dynamics consists of coupled ODEs and MDE:
\begin{eqnarray}\label{eq:ODE-MDE_exmp}
\dot{S}=-\frac{S}{N} \int_{\R} \beta(\alpha) \ dI(\alpha),\nonumber\\
\dot{I}=V_\varphi[I]+ \frac{S}{N} \beta(\alpha) I -\nu(\alpha) I,\nonumber\\
\dot{R}=\int_{\R}\nu(\alpha) \ dI(\alpha).
\end{eqnarray}

% {\color{blue}{

% The approximated solution $\mu^M$ to the coupled ODE-MDE system:
% \[S^0=S_0\]
% \[I^M(0) =\mathcal{A}_M^x(I_0);\]
% \[I^M((k+1)\Delta_M) = \sum\limits_{i,j} m_{i,j}^{v}(V[I^M(k\Delta_M)]) \delta_{x_i+\Delta M v_j}+\Delta_M \mathcal{A}_M^x(s[I^M(k\Delta_M)])\]
% \[s[\mu^M(k\Delta_M)])=\left| \beta(\alpha)\frac{S^M}{N} \right| I^M(k\Delta_N) -\nu(\alpha) I^M(k\Delta_M)\]
% \[S^M=S^{M-1} +h \left(-\beta(\alpha)\frac{S^{M-1}}{N}|I|\right)\]

% For any $\tau \in [0, \Delta M]$, \[I^M((k\Delta_M+\tau) = \sum\limits_{i,j} m_{i,j}^{v}(V[I^M(k\Delta_M)]) \delta_{x_i+\tau v_j}+\tau \mathcal{A}_M^x(s[I^M(k\Delta_M)])\]
% }
% \todo[inline]{Note that the approximated solution dose not depend on $R$. Next step is to pass the limit and show the existence of solutions (using Arzela-Ascoli theorem) }

% }
\begin{prop}
Consider the system \eqref{eq:ODE-MDE_exmp}, with
$V_\varphi$ given by \eqref{eq:MDE-cum-phi},
and assume $\varphi$, $\beta$ and $\nu$ Lipschitz continuous and bounded.
Then there exists a Lipschitz semigroup of solutions.
Moreover there exists unique solutions compatible with
Euler-LAS approximations.
\end{prop}

\begin{proof}
We first show that assumptions in the statement of Theorem \ref{th:ODE-MDE} are satisfied.
The bounds on $\varphi$, $\beta$ and $\nu$ implies
that all solutions to a Cauchy problem is uniformly bounded in time. Therefore the assumptions of Theorem \ref{th:ODE-MDE}
can be verified on bounded sets of $\R^2\times\M(\R)$
(the space of $(S,I,R)$) for the metric given by
the Euclidean norm on $\R^2$ and $W^g$ on $\M(\R)$.
Indeed we show that all assumptions are valid on the whole space except (OM:Lip-V) and (OM:Lip-s).

Since $\varphi$ is increasing, it is uniformly bounded
by {\small$\max\{|\varphi(0)|,|\varphi(1)|\}$, thus
(H:bound)} is verified for $V_\varphi$.

From the Lipschitz continuity and boundedness of $\beta$ and $\nu$ we get:
\[
\left| \int_{\R} \beta(\alpha) \ dI_1(\alpha)-
\int_{\R} \beta(\alpha) \ dI_2(\alpha)\right|
\leq L(\beta) W^g(I_1,I_2)+M(\beta),
\]
where $L(\beta)$ is the Lipschitz constant of $\beta$ and
$M(\beta)>0$ a uniform bound. Moreover, the same
estimate holds for $\nu$.
Since the other terms of the ODEs are Lipschitz,
we conclude that (OM:Lip-g) is satisfied.

Now, consider $I_1,I_2\in\M(\R)$ and fix
$\tilde{I}_1\leq I_1$ and $\tilde{I}_2\leq I_2$
realizing $W^g(I_1,I_2)$.
Choose {\small$\tilde{T}\in P(V_\varphi(\tilde{I_1}),V_\varphi(\tilde{I_2}))$}
such that $\pi_{13}\# \tilde{T}=\min \{G_{\tilde{I}_1}(x), G_{\tilde{I}_2}(y)\} \in P^{opt}(\tilde{I_1},\tilde{I_2})$
(see Theorem 2.18 and Remark 2.19 of \cite{Villani03}).
Since $|\tilde{I}_1|=|\tilde{I}_2|$, by definition of
$V_\varphi$ we get:
\begin{equation}\label{eq:intzero}
\int |v-w|\ d\tilde{T}(x,v,y,w) =0.
\end{equation}
Define $\tilde{V}_1\leq V_\varphi(I_1)$ such that  $\pi_1(\tilde{V}_1)=\tilde{I}_1$
and $\int_v v\, d\tilde{V}_1(x,v)$ is minimum among the $V\in\M(T\R)$
such that  $\pi_1(V)=\tilde{I}_1$.
Similarly define $\tilde{V}_2$. In simple words we select the minimum speeds.
Now choose $T\in P(\tilde{V}_1,\tilde{V}_2)$ such that
$\pi_{13}\# T=\min \{G_{\tilde{I}_1}(x), G_{\tilde{I}_2}(y)\}$.
Notice that $\pi_{13}\# T\in P^{opt}(\tilde{I_1},\tilde{I_2})$.

By definition:
\[
\WW^g(V_\varphi(I_1),V_\varphi(I_2))
\leq \int |v-w| \, dT(x,v,y,w),
\]
indeed $\pi_{13}\# T\in P^{opt}(\tilde{I}_1,\tilde{I}_2)$ and $(\tilde{I}_1,\tilde{I}_2)$ realizes $W^g(I_1,I_2)$.\\
Now choose any $T_1\in P(\tilde{V}_1,V_\varphi(\tilde{I_1}))$
such that $\pi_{13}\# T_1=Id$, the transference plan
corresponding to the identity map $Id(x)=x$.
The existence of such
$T_1$ comes from the fact that $\pi_1(\tilde{V}_1)=
\pi_1(V_\varphi(\tilde{I_1}))=\tilde{I_1}$.
We choose similarly $T_2\in P(V_\varphi(\tilde{I_2}),\tilde{V}_2)$
such that $\pi_{13}\# T_2=Id$.
Notice that $T=T_1\circ \tilde{T}\circ T_2$, thus
we can write:
\begin{eqnarray}
& \int |v-w| \, dT(x,v,y,w)\leq\nonumber\\
& \int |v-w| \, dT_1(x,v,y,w)+
 \int |v-w| \, d\tilde{T}(x,v,y,w)+
 \int |v-w| \, dT_2(x,v,y,w)\nonumber\\
&\doteq A_1+\tilde{A}+A_2.\nonumber
\end{eqnarray}
From \eqref{eq:intzero} it holds $\tilde{A}=0$. Moreover we have:
\begin{equation}\label{eq:A_1A_2}
A_1\leq
\int |v-w|\, d(V_\varphi(I_1)\times V_\varphi(\tilde{I}_1)),\quad
A_2\leq
 +\int |v-w| d(V_\varphi(I_2)\times V_\varphi(\tilde{I}_2)).
\end{equation}
Now we notice that, since $\varphi$ is Lipschitz continuous,
we have:
\[
|\varphi(G_{I_1}(x))-\varphi(G_{\tilde{I}_1}(x))|
\leq L(\varphi) |G_{I_1}(x)-G_{\tilde{I}_1}(x)|
\]
where $L(\varphi)$ is a Lipschitz constant for $\varphi$, thus:
\begin{eqnarray}\label{eq:est-intv}
\int |v-w|\, d(V_\varphi(I_1)\times V_\varphi(\tilde{I}_1))\leq
L(\varphi) \int |x|\, d(G_{I_1}(x)-G_{\tilde{I}_1}(x))\nonumber\\
\leq L(\varphi) \left(sup_{x\in Supp(I_1)}|x|\right)
|I_1-\tilde{I}_1|.
 \end{eqnarray}
The same estimate holds for $I_2$ and $\tilde{I}_2$.
Using \eqref{eq:A_1A_2} and
\eqref{eq:est-intv}, we get that
$A_1$ and $A_2$ are bounded
by $C(|I_1-\tilde{I}_1|+|I_2-\tilde{I}_2|)$
for some positive $C>0$ bounded over bounded sets.
Moreover,  $W^g(I_1,I_2) \geq |I_1-\tilde{I}_1|+|I_2-\tilde{I}_2|$.
Therefore (OM:Lip-V) holds on bounded sets.

The boundedness of $\beta$ and $\nu$ imply (OM:bound-s).

Notice that the source term is given by
$(\beta\frac{S}{N}-\nu)I$, thus:
\begin{eqnarray}
W^g(s((S_1,R_1),I_1),s((S_2,R_2),I_2))\leq
|I_1||(S_1,R_1)-(S_2,R_2)|\nonumber\\
+|(S_2,R_2)|W^g(I_1,I_2).\nonumber
 \end{eqnarray}
Therefore (OM:Lip-s) is verified on bounded subsets
of $\R^2\times\M(\R)$. This concludes the proof
of the first statement.

Now notice that all Euler-LAS approximations give rise
to a Cauchy sequence in $\R^2\times \M(\R)$
for the metric given by
the Euclidean norm on $\R^2$ and $W^g$ on $\M(\R)$.
Therefore, we conclude by completeness of the same metric.
\end{proof}

\begin{lemma}
Consider the system \eqref{eq:ODE-MDE_exmp}, with
$V_\varphi$ given by \eqref{eq:MDE-cum-phi},
and assume $\varphi$, $\beta$ and $\nu$ Lipschitz continuous and bounded.
Then the quantity $S+|I|+R$ is conserved along
solutions.
\end{lemma}
\begin{proof}
It is sufficient to notice that the total mass of the positive source for $I$
in \eqref{eq:ODE-MDE_exmp} coincides with the absolute value
of the right-hand side for $S$. Similarly the total mass of the negative source for $I$ in coincides with the right-hand side for $R$.
\end{proof}

\begin{example}
Consider the model \eqref{eq:ODE-MDE_exmp} and assume
that $\beta$, $\nu$ are constant, i.e. independent of
$\alpha$. Then we have:
\[
\dot{S}=-\frac{S}{N}\,\beta\, |I|,\quad \dot{R}=\nu\, |I|.
\]
Since $V_\varphi$ does not change the total mass of $I$,
defining $m(t)=|I(t)|$, we get:
\[
\dot{m}(t)=\frac{S}{N}\,\beta\,m-\nu\,m,
\]
thus the triplet $(S,m,R)$ satisfy the classical SIR model.
\end{example}

\begin{example}
Consider the model \eqref{eq:ODE-MDE_exmp} with $V$ as in Example
\ref{ex:1} and $I(0)=|I(0)|\delta_0$, thus the mass of $I$ is split in two parts traveling with velocity $-1$ and $+1$ respectively.
Assume $\beta(\alpha)=\beta(-\alpha)$ and $\nu(\alpha)=\nu(-\alpha)$,
then we get $\int_\R \beta(\alpha) dI(t)(\alpha)=\beta(t)$
and  $\int_\R \nu(\alpha) dI(t)(\alpha)=\nu(t)$. Therefore,
defining $m(t)=|I(t)|$, the solution
to \eqref{eq:ODE-MDE_exmp} satisfies the time-dependent SIR-model:
\[
\dot{S}=-\frac{S}{N}\,\beta(t)\, m(t),\quad
\dot{m}= \frac{S}{N}\,\beta(t)\,m-\nu(t)\,m,\quad
\dot{R}=\nu(t)\, m.
\]
\end{example}

%\section{Age-structured ODE-MDE}

\section*{Acknowledgements}
The authors would like to acknowledge the support of the NSF CMMI project \# 2033580 ``Managing pandemic by managing mobility"
in collaboration with Cornell University and Vanderbilt University,
and the support of the Joseph and Loretta Lopez Chair endowment.

%\bibliography{AIMS}
%\bibliographystyle{plain}

\medskip
Received May 2021; revised September 2021; early access March 2022.
\medskip

\end{document}